\documentclass[11pt,twoside,a4paper]{amsart}
\usepackage[utf8]{inputenc}
\usepackage{amsmath,amssymb}
\usepackage{mathrsfs}
\usepackage{amsfonts}
\usepackage{mathtools}
\usepackage{amsthm}
\usepackage{tikz-cd}
\newcommand{\0}{\mathscr{O}}

\newcommand{\M}{\mathscr{M}}

\theoremstyle{plain} 
\newtheorem{thm}{Theorem}[section] 
\newtheorem*{Thm}{Theorem}
\newtheorem{cor}[thm]{Corollary} 
\newtheorem{lem}[thm]{Lemma} 
\newtheorem{prop}[thm]{Proposition}
\theoremstyle{definition} 
\newtheorem{defn}[thm]{Definition}

\theoremstyle{remark} 
\newtheorem{rmk}[thm]{Remark}

\begin{document}

\title[Towards generic base-point-freeness for hyperkähler manifolds of generalized Kummer type]{Towards generic base-point-freeness for hyperkähler manifolds of generalized Kummer type}  
\author[M. Varesco]{Mauro Varesco}
\address{Mathematisches Institut, Universitat Bonn, Endenicher Allee 60, 53115 Bonn, Germany}
\email{varesco@math.uni-bonn.de}

\begin{abstract}
We study base-point-freeness for big and nef line bundles on hyperkähler manifolds of generalized Kummer type: For $n\in \{2,3,4\}$, we show that, generically in all but a finite number of irreducible components of the moduli space of polarized $\mathrm{Kum}^n$-type varieties, the polarization is base-point-free. We also prove generic base-point-freeness in the moduli space in all dimensions if the polarization has divisibility one.
\end{abstract}
\maketitle

\section*{Introduction}
The starting point for the analysis presented in this article is the following observation in the context of Fujita’s conjecture: 
Given a K3 surface $X$ and an ample line bundle $H\in \mathrm{Pic}(X)$, Mayer proved in \cite{M} that the line bundle $2H$ is base-point-free. This result is stronger than what one gets when applying Fujita's conjecture to the case of K3 surfaces. Indeed, the conjecture predicts that the line bundle $3H$ is base-point-free, see e.g. \cite[Conj.\ 10.4.1]{La}.
This suggests that it might be interesting to study questions related to big and nef line bundles on hyperkähler manifolds.

\smallskip
Let $\M_{d,t}^n$ be the moduli space of hyperkähler manifolds of $\mathrm{Kum}^n$-type with a polarization of square $2d$ and divisibility $t$. 
In this article, we study base-point-freeness for big and nef line bundles on hyperkähler manifolds of $\mathrm{Kum}^n$-type: 
In low dimenision, that is $n\in \{2,3,4\}$, we prove that, for all but a finite number of choices of $d$ and $t$, the polarization of the generic element of $\M_{d,t}^n$ is base-point-free:

\begin{Thm} [Theorem \ref{Final Theorem}]
Let $n\in\{2,3,4\}$, and let $d$ and $t$ be positive integers such that the moduli space $\M_{d,t}^n$ is non-empty. Let $A$ be the set of triples
\[A\coloneqq\{(2,1,2),(3,4,2),(3,28,8),(3,92,8),(4,3,2),(4,20,5),(4,55,10)\}.\] 
Then, if $(n,d,t)\not\in A$, the polarization of a general pair in $\M_{d,t}^n$ is base-point-free.
\end{Thm}
In the case where the divisibility of the polarization is one, we prove generic base-point-freeness in all dimensions:
\begin{Thm}[Theorem \ref{thm bpf for divisibility 1}]
The polarization of the general pair of $\M_{d,1}^n$ is base-point-free for all $d>0$ and 
$n>1$.
\end{Thm}
In the article \cite{R} by Rie\ss ~, a similar statement is proven for hyperkähler manifolds which are deformation equivalent to the Hilbert scheme of points on a K3 surface: There, the author proves generic base point freeness in the moduli space $\mathscr{F}_{d,t}^2$ parametrizing pairs $(X,L)$ where $X$ is a hyperkähler manifold of $\mathrm{K3}^{[2]}$-type and $L$ is a polarization on $X$ of square $2d$ and divisibility $t$.
% and proves in particular the following result:
% \begin{Thm}\cite[Thm.\ 7.2]{R}
% Let $d$ and $t$ be positive integers such that the moduli space $\mathscr{F}_{d,t}^2$ is non-empty. Then, the polarization of a generic pair in $\mathscr{F}_{d,t}^2$ is base-point-free.
% \end{Thm}
This result has then been extended by Debarre in \cite{D}: There, the author studies in all dimensions the case where the divisibility of the polarization is one or two, showing that, if $t=1$ and $d\geq n-1$ or if $t=2$ and $d\geq n+3$, the polarization of a generic pair $(X,L)\in \mathscr{F}^n_{d,t}$ is base-point-free.
% \begin{Thm}\cite[Cor.\ 3.9]{D}
% Let $n, t$ and $d$ be positive integers with $n\geq 2$ and $t\in \{1,2\}$. Then, for a generic element $(X,L)\in\mathscr{F}_{d,t}^n$ the following holds:
% \begin{itemize}
%     \item[(\textit{i})] For $t = 1$, the line bundle $L$ is base-point-free if $d \geq n - 1$.
%     \item[(\textit{ii})] For $t = 2$, the line bundle $L$ is base-point-free if $d \geq n + 3$.
% \end{itemize}
% \end{Thm}
In dimension six, generic base-point-freeness for $\mathrm{K3}^{[3]}$-type varieties is proven by Agostini and Rie\ss~ in the draft \cite{AR}.
% \begin{Thm}\cite{AR}
% Let $d,t$ be positive integers such that the moduli space $\mathscr{F}_{d,t}^3$ is non-empty. Then, for a generic pair $(X,L)\in\mathscr{F}_{d,t}^3$, the polarization $L$ is base-point-free.
% \end{Thm}
For a description of the divisorial component of the base locus of big and nef line bundles on hyperkähler manifolds of $\mathrm{K}3^{[n]}$-type and of $\mathrm{Kum}^n$–type see the article \cite{R18} by Rie\ss ~. There the author  provides in particular a criterion for big and nef line bundles to have base divisor. 
% Denoting by $\chi(X,H)$ the Euler characteristic of a line bundle, $H\in \mathrm{Pic}(X)$, the following holds:
% \begin{Thm}\cite[Thm.\ 4.7]{R18}
%  Let $X$ be a hyperkähler manifold of $\mathrm{K}3^{[n]}$-type or of $\mathrm{Kum}^n$–type. Then, a big and  nef line bundle $H \in \mathrm{Pic}(X)$ has a non-trivial base divisor
% if and only if there exists an irreducible reduced divisor $F$ with negative square such that
% $H$ is of the form $H = mL+F$, where $m \geq 2$, $L$ is a primitive movable line bundle with
% $q(L) = 0$ and $q(L, F) > 0$, such that $\chi(X,H) = \binom{m+n}{n}$. In this case, $F$ is exactly
% the fixed divisor of $H$.
% \end{Thm}

We are the first, to our knowledge, to study generic base-point-freeness for the moduli space of polarized $\mathrm{Kum}^n$-type varieties. Nevertheless, most ideas and techniques that we have used in this paper are inspired by the work of the aforementioned authors. The structure of this article is the following:

\smallskip

In Section \ref{Section general conventions}, we recall the definition of \textit{hyperkähler manifold}, and we collect some results on the lattice structure of the integral second cohomology group of hyperkähler manifolds of $\mathrm{Kum}^n$-type.

\smallskip
In Section \ref{section the moduli space}, we deduce from a result in \cite{O} that the moduli space $\M_{d,t}^n$ is connected when non-empty for $n\in \{2,3, 4\}$. Since it is known that this space is always smooth, we conclude that it is irreducible. We then recall the proof of the fact that base-point-freeness is an open property in families of polarized hyperkähler manifolds. This will lead us to a useful criterion for generic base-point-freeness in the moduli space.

\smallskip
Section \ref{section divisibility one} is devoted to study the case where the divisibility of the polarization is one. We deal with it separately since the technique that we use differs from the ones we use in the other cases. In particular, we prove that any effective line bundle $L$ on an abelian surface $T$ induces a  base-point-free line bundle on $\mathrm{Kum}^n(T)$ for all $n>1$. We will use this to prove generic base-point-freeness in the moduli space $\M^n_{d,1}$ for all $n>1$ and $d>0$.

\smallskip
In Section \ref{Sec: tautological bundles}, we recall the notions of tautological bundles on the Hilbert scheme of points and of $k$-very ampleness. These are key ingredients for the proof of generic base-point-freeness in the remaining cases.

\smallskip

Finally, in Section \ref{last section}, we deal with polarizations with higher divisibility. We prove that, for all but a finite number of cases, generic base-point-freeness holds for $\M^n_{d,t}$ for $n\in\{2,3,4\}$. This will conclude the proof of the main result of this article.

\section*{Acknowledgements}
This work is part of my Master's thesis at ETH Z\"{u}rich. I would like to thank Ulrike Rie\ss ~ for suggesting me this topic, for supervising my thesis, and for her continuous advice and support.
I am also grateful for the financial support provided by ETH Foundation during my master's, and by ERC Synergy Grant HyperK
(Grant agreement No.\ 854361) during the completion and editing of the article. In particular, I express my gratitude to Daniele Agostini, Fabrizio Anella, and Daniel Huybrechts for reading a preliminary version of this paper.

\section{Hyperkähler manifolds}
\label{Section general conventions}
Let us begin by recalling the definition of hyperkähler manifold. For an introduction to the subject see \cite{GHJ}. 
\begin{defn}
A \textit{hyperkähler manifold} is a simply-connected compact K\"ahler manifold $X$, such that $H^0(X,\Omega^2_X)$ is generated by an everywhere non-degenerate holomorphic two-form.
\end{defn}
A particular family of examples of hyperkähler manifolds can be constructed as follows: Let $T$ be an abelian surface and let $T^{[n+1]}$ be the Hilbert scheme of $n+1$ points on $T$. Then, the fibre over $0$ of the natural morphism $T^{[n+1]}\rightarrow T$ induced by the sum is a hyperkähler manifold of dimension $2n$. It is called \textit{generalized Kummer variety}, and it is denoted by $\mathrm{Kum}^n(T)$.
In this paper, we focus on varieties which are deformation equivalent to some generalized Kummer variety. These varieties are still hyperkähler manifolds and are called $\mathrm{Kum}^n$-type varieties.

\smallskip
Recall that the second cohomology group $H^2(X,\mathbb{Z})$ of a hyperkähler manifold $X$ has a natural lattice structure induced by the Beauville--Bogomolov quadratic form $q$. Talking about the lattice $(H^2(X,\mathbb{Z}),q)$, we will use the following definitions:
\begin{defn}
Let $(\Lambda, q)$ be a lattice and let $\alpha$ be an element in $\Lambda$, then:
\begin{itemize}
 \item[(\textit{i})] The element $\alpha$ is called \textit{primitive}, if it is not a non-trivial multiple of another element, i.e., $\alpha=k\cdot \alpha'$ for some $k\in \mathbb{Z}$ and $\alpha'\in \Lambda$ implies $k=\pm 1$.
 \item[(\textit{ii})] The \textit{divisibility of }$\alpha$, $\mathrm{div}(\alpha)$, is the multiplicity of the element $q(-,\alpha)\in\Lambda^{\vee}$, i.e., div$(\alpha)=m$ if $m$ spans the ideal $\{q(B,\alpha)|\; B\in\Lambda\}\subseteq \mathbb{Z}$.
 \item[(\textit{iii})] Let $\{e_i\}$ be a $\mathbb{Z}$-basis of the lattice $\Lambda$. The \textit{discriminant} of the lattice $(\Lambda, q)$ is defined as the determinant of the matrix $(q(e_i,e_j))_{i,j}$.
\end{itemize}
\end{defn}
As the discriminant of a lattice does not depend on the choice of the basis, we see that the divisibility of a primitive element divides the discriminant of the lattice.
\begin{rmk}
\label{ref: rmk divisibility is invariant under deformations}
If $X$ is a hyperkähler manifold, then $H^1(X,\0_X)=0$. Therefore, using the exponential sequence, one identifies the Picard group of $X$ with a subgroup of the lattice $H^2(X,\mathbb{Z})$.
Given an element $L\in \mathrm{Pic}(X)$, we will consider its divisibility with respect to the lattice $(H^2(X,\mathbb{Z}),q)$. As a consequence, the divisibility of any element of Pic$(X)$ is invariant under deformation of $X$.
\end{rmk}

In the case of hyperkähler manifolds of $\mathrm{Kum}^n$-type, the lattice structure on the second cohomology has been described: 

\begin{prop}\cite[Prop.\ 8]{B}
\label{ref: prop lattice Kum type ISV} 
Let $X\coloneqq\mathrm{Kum}^n(T)$ for some abelian surface $T$. Then:
\begin{itemize}
 \item[(\textit{i})] The lattice $(H^2(X,\mathbb{Z}),q)$ has the following orthogonal decomposition:
 \[(H^2(X,\mathbb{Z}),q)=(H^2(T,\mathbb{Z}),\cup)\oplus (\mathbb{Z}\delta),\]
 where $\cup$ is the cup product on $H^2(T,\mathbb{Z})$ and $\delta$ is the restriction to $X$ of the divisor $\frac{1}{2}E$, where $E$ is the exceptional divisor on $T^{[n+1]}$ of the Hilbert--Chow morphism. In particular, $q(\delta)=-2n-2$.
 %\item There is an orthogonal decomposition
 %\[\mathrm{NS}(X)=\mathrm{NS}(T)\oplus\mathbb{Z}\delta.\]
 \item[(\textit{ii})] Let $Y$ be a hyperkähler manifold of Kum$^n$-type which is deformation equivalent to $X$. Then, the lattice $(H^2(Y,\mathbb{Z}),q)$ has the following orthogonal decomposition:
 \[(H^2(Y,\mathbb{Z}),q)\simeq(H^2(X,\mathbb{Z}),q)=(H^2(T,\mathbb{Z}),\cup)\oplus (\mathbb{Z}\delta).\]
\end{itemize}
\end{prop}

\begin{rmk}
\label{rmk: delta on hilb}
By abuse of notation, we will use $\delta$ to denote both the divisor $\frac{1}{2}E$ on the Hilbert scheme and its restriction on the generalized Kummer variety as in Proposition \ref{ref: prop lattice Kum type ISV}. Similarly, given a line bundle $L$ on an abelian surface $T$, we will denote by $L_n$ both the induced divisor on $T^{[n+1]}$ and its restriction to $\mathrm{Kum}^n(T)$. It will be clear from the context which of the two we are referring to.
\end{rmk}

\begin{rmk}
\label{rmk: divisibility divides GCD}
Let $X=\mathrm{Kum}^n(T)$ for some abelian surface $T$. Proposition \ref{ref: prop lattice Kum type ISV} implies that any element $\alpha\in H^2(X,\mathbb{Z})$ can be expressed as $\alpha=a\lambda_n+b\delta$ for a primitive element $\lambda\in H^2(T,\mathbb{Z})$ and $a,b\in \mathbb{Z}$. Since $(H^2(T,\mathbb{Z}),\cup)$ is a unimodular lattice, the divisibility of $\alpha$ is
\[\mathrm{div}(\alpha)=\mathrm{gcd}(a,2(n+1)b).\]
\end{rmk}

\section{The moduli space}
\label{section the moduli space}
Let $\mathscr{M}_{d,t}^n$ be the moduli space which parametrizes pairs $(X,L)$ where $X$ is a hyperkähler manifold of Kum$^n$-type and $L$ is a primitive and ample line bundle on $X$ with $q(L)=2d$ and div$(L)=t$.

\begin{rmk}
\label{rmk: moduli space empty}
Note that, by Remark \ref{rmk: divisibility divides GCD}, the moduli space $\M^n_{d,t}$ is empty
if $t\not|\gcd(2d,2n+2)$.
\end{rmk}

In \cite{O}, the author characterizes the number of connected components of the moduli space
$\mathscr{M}_{d,t}^n$. Let us fix the necessary notation to recall that result. 
Let $n,d,t>0$ be integers such that $t\, |\, \gcd(2d,2n+2)$ and set
\[d_1\coloneqq \frac{2d}{\gcd(2d,2n+2)},\quad n_1\coloneqq \frac{2n+2}{\gcd(2d,2n+2)},\quad
g\coloneqq \frac{\gcd(2d,2n+2)}{t},\]
\[w\coloneqq \gcd(g,t),\qquad g_1\coloneqq \frac{g}{w},\qquad t_1\coloneqq \frac{t}{w}.\]
Let $\phi$ be the Euler function, i.e., the function that associates to a positive integer $l$ the number of integers in $\{1,\ldots,l\}$ which are coprime with $l$. Moreover, let $\rho(l)$ be the number of primes in the factorization of $l$. For $w$ and $t_1$ as above, write $w=w_{+}(t_1)w_{-}(t_1)$, where $w_{+}(t_1)$ is the product of all powers of primes dividing $\gcd(w,t_1).$ Finally, denote by $|\M^n_{d,t}|$ the number of connected components of $\M^n_{d,t}$.

\begin{thm}\cite[Thm.\ 5.3]{O}
\label{thm conn comp of mod space}
With the above notation we have:
\begin{enumerate}
 \item $|\M^n_{d,t}|=w_+(t_1)\phi(w_{-}(t_1))2^{\rho(t_1)-1}$ if $t>2$ and one of the following holds:
 \begin{enumerate}
 \item $g_1$ is even, $\gcd(d_1,t_1)=1=\gcd(n_1,t_1)$ and $-d_1/n_1$ is a quadratic residue mod $t_1$;
 \item $g_1, t_1$, and $d_1$ are odd, $\gcd(d_1,t_1)=1=\gcd(n_1,2t_1)$ and $-d_1/n_1$ is a quadratic residue mod $2t_1$;
 \item $g_1, t_1$, and $w$ are odd, $d_1$ is even, $\gcd(d_1,t_1)=1=\gcd(n_1,2t_1)$ and $-d_1/4n_1$ is a quadratic residue mod $t_1$.
 \end{enumerate}
 \item $|\M^n_{d,t}|=w_+(t_1)\phi(w_{-}(t_1))2^{\rho(t_1/2)-1}$ if $t>2$, $g_1$ is odd, $t_1$ is even, $\gcd(d_1,t_1)=1=\gcd(n_1,2t_1)$ and $-d_1/n_1$ is a quadratic residue mod $2t_1$.
 \item $|\M^n_{d,t}|=1$ if $t\leq 2$ and one of the following holds:
 \begin{enumerate}
 \item $g_1$ is even, $\gcd(d_1,t_1)=1=\gcd(n_1,t_1)$ and $-d_1/n_1$ is a quadratic residue mod $t_1$;
 \item $g_1, t_1$, and $d_1$ are odd, $\gcd(d_1,t_1)=1=\gcd(n_1,2t_1)$ and $-d_1/n_1$ is a quadratic residue mod $2t_1$;
 \item $g_1, t_1$, and $w$ are odd, $\gcd(d_1,t_1)=1=\gcd(n_1,2t_1)$ and $-d_1/4n_1$ is a quadratic residue mod $t_1$;
 \item $g_1$ is odd, $\gcd(d_1,t_1)=1=\gcd(n_1,2t_1)$ and $-d_1/n_1$ is a quadratic residue mod $2t_1$.
 \end{enumerate}
 \item $|\M^n_{d,t}|=0$ otherwise.
\end{enumerate}
\end{thm}

In the case $n\in \{2, 3, 4\}$, Theorem \ref{thm conn comp of mod space} implies the connectedness of $\M^n_{d,t}$:
\begin{cor}
\label{cor: moduli is connected}
Let $n\in\{2,3, 4\}$ and let $t,d>0$ be integers such that $\M^n_{d,t}$ is non-empty. Then, $\M^n_{d,t}$ is connected.
\begin{proof}
We do the computation only for $n=2$, the  other cases can be checked in the same way. By Remark \ref{rmk: moduli space empty}, the moduli space $\M^n_{d,t}$ is non-empty by assumption only if $t$ divides $\gcd(2d,2n+2)$. For $n=2$, we then have to consider only the cases where $t \in \{1,2,3,6\}$.
\smallskip

If $t=1,2$, the moduli space $\M^2_{d,t}$ is connected by the case $(3)$ of Theorem \ref{thm conn comp of mod space}.

\smallskip
If $t=3$, then $g=\frac{\gcd(2d,6)}{3}=2$, $w=\gcd(2,3)=1, g_1=2$, and $t_1=3$. Therefore, since $g_1$ is even and since we are assuming that $\M^2_{d,t}$ is non-empty, we are in the case $(1.a)$ of Theorem $\ref{thm conn comp of mod space}$. From $w_+(t_1)=w_{-}(t_1)=1$, we then deduce that
 \[
 |\M^2_{d,3}|=1\phi(1)2^{\rho(3)-1}=1.
 \]
 
 If $t=6$, then $g=\frac{\gcd(2d,6)}{6}=1$, $w=\gcd(1,6)=1, g_1=1$, and $t_1=6$. Therefore, since $g_1$ is odd and $t_1$ is even and since we are assuming that $\M^2_{d,t}$ is non-empty, we are in the case $(2)$ of Theorem $\ref{thm conn comp of mod space}$. As $w_+(t_1)=w_{-}(t_1)=1$, we deduce that
 \[
|\M^2_{d,6}|=1\phi(1)2^{\rho(6/2)-1}=1. \qedhere
 \]
 \end{proof}
\end{cor}
It is a standard fact that the moduli space $\M^n_{d,t}$ is smooth:
\begin{prop}\cite[Sec.\ 1.14]{Huy}
\label{prop M^n is smooth}
For every $n>1$ and every positive integers $d$ and $t$, the moduli space $\M^n_{d,t}$ is smooth.
\begin{proof}
Let $(X,L)$ be a point of $\M^n_{d,t}$, and let $\mathrm{Def}(X,L)$ be the universal deformation of the pair $(X,L)$ as in \cite[Sec.\ 1.14]{Huy}. Locally around $(X,L)$, the moduli space $\M^n_{d,t}$ is isomorphic to the open subset of $\mathrm{Def}(X,L)$ where the line bundle $L$ remains ample. The statement then follows from the fact that $\mathrm{Def}(X,L)$ is a smooth hypersurface of the universal deformation space $\mathrm{Def}(X)$.
%Let $n, t,$ and $d$ be positive integers such that $\M^n_{d,t}$ is non-empty, and let $(X,L)\in \M^n_{d,t}$. It is possible to find a simply connected open neighbourhood $U$ of $(X,L)$ admitting a universal family, which can be equipped with a marking. Using the local Torelli theorem (see \cite[Prop.\ 22.11]{GHJ}), one can prove that, possibly after shrinking $U$, the neighbourhood $U$ of $(X,L)$ can be identified via the period map $\mathcal{P}\colon\mathrm{Def}(X)\rightarrow Q_{\Lambda}$ with an open subset of $Q_{L^{\bot}}\subseteq Q_{\Lambda}$ where $L^{\bot}$ is the orthogonal complement of $L$ in $\Lambda$. Since $Q_{L^{\bot}}$ is a smooth quadric in $\mathbb{P}(L^{\bot}\otimes \mathbb{C})$, this shows that $\M^n_{d,t}$ is smooth at every point.
\end{proof}
\end{prop}
As an immediate corollary we have the following:
\begin{cor}
\label{cor:conn comp are irreducible}
Every connected component of $\M^n_{d,t}$ is irreducible. In particular, for $n\in\{2,3,4\}$, the moduli space $\M^n_{d,t}$ is irreducible whenever it is non-empty.
\end{cor}

In the following proposition, we recall the standard fact that base-point-freeness is an open property in families of hyperkähler manifolds:

\begin{prop}
\label{prop locally free is an open property in families}
Let $\zeta\colon \mathscr{X}\rightarrow S$ be a family of hyperkähler manifolds over a connected base $S$. Let $\mathscr{L}$ be a line bundle on $\mathscr{X}$ such that $q(\mathscr{L}_s)>0$ for all $s\in S$, where $\mathscr{L}_s$ denotes the restriction of $\mathscr{L}$ to the fibre $\mathscr{X}_s$. Then, if $\mathscr{L}_0$ is base-point-free for some $0\in S$, then $\mathscr{L}_s$ is base-point-free for all $s$ in an open neighbourhood of $0\in S$.
\begin{proof}
Let us begin by reducing to the case where $\mathscr{L}$ is a line bundle on $\mathscr{X}$ whose push-forward $\zeta_*\mathscr{L}$ is a free sheaf of rank $h^0(\mathscr{L}_0)$:
By assumption, the line bundle $\mathscr{L}_0$ is base-point-free and satisfies $q(\mathscr{L}_0)>0$. This implies that $\mathscr{L}_0$ is big and nef by the Beauville--Fujiki Relation \cite[Prop.\ 23.14]{GHJ}. Therefore, $h^i(\mathscr{L}_0)=0$ for all $i>0$ by the Kodaira vanishing theorem. The semicontinuity theorem then implies that $h^i(\mathscr{L}_s)=0$ for all $i>0$ and all $s$ in a neighbourhood $U$ of $0\in S$, see \cite[Thm.\ III.12.8]{Ha}. Since the Euler characteristic $\chi(\mathscr{L}_s)$ does not depend on $s\in S$ by flatness, $h^0(\mathscr{L}_s)$ is constant in the neighbourhood $U$.
Therefore, the push forward $(\zeta|_{\zeta^{-1}U})_*(\mathscr{L}|_{\zeta^{-1}U})$ is a locally free sheaf of rank $m=h^0(\mathscr{L}_0)$ by Grauert theorem, see \cite[Cor.\  III.12.9]{Ha}.
After substituting $S$ with a suitable open neighbourhood of $0$ and $\mathscr{X}$ with the preimage via $\zeta$ of this neighbourhood, we may assume that $\mathscr{L}$ is a line bundle on $\mathscr{X}$ such that \[\zeta_*\mathscr{L}\simeq\0_S^{\oplus m}.\]
Let $\{e_1,\ldots,e_m\}$ be the standard basis of the $H^0(\0_S)$-module $H^0(\0_S^{\oplus m})$. Since $H^0(\mathscr{L})\simeq H^0(\0_S^{\oplus m})$, we may view $\{e_1,\ldots,e_m\}$ as a basis of $H^0(\mathscr{L})$. The base locus of $|\mathscr{L}|$ is then
\[\mathrm{BL}(|\mathscr{L}|)=\displaystyle\bigcap_{i=1}^m V(e_i)\subseteq\mathscr{X}.\]
The set $\zeta(\mathrm{BL}(|\mathscr{L}|))\subseteq S$ is closed since $\zeta$ is a proper map and $\mathrm{BL}(|\mathscr{L}|)$ is closed. By assumption, the line bundle $\mathscr{L}_0$ is base-point-free. Therefore, the element $0\in S$ does not belong to $\zeta(\mathrm{BL}(|\mathscr{L}|)$ since $H^0(\mathscr{L})$ surjects onto $H^0(\mathscr{L}_0)$ via restriction.
This allows us to conclude that $W\coloneqq S\setminus \zeta(\mathrm{BL}(|\mathscr{L}|)$ is an open neighbourhood of $0\in S$ such that for all $s\in W$ the line bundle $\mathscr{L}_s$ is base-point-free.
\end{proof}
\end{prop}
Let us deduce from Proposition \ref{prop locally free is an open property in families} the following criterion for generic base-point-freeness in the moduli space $\M_{d,t}^n$:
\begin{prop}
\label{prop sufficient condition to prove generic bpf}
Let $n>1$, and let $d$ and $t$ be positive integers such that the moduli space $\M_{d,t}^n$ is irreducible. Assume that there exists a pair $(X,L)$ where $X$ is a $\mathrm{Kum}^{n}$-type variety and $L$ is a base-point-free line bundle on $X$ with $q(L)=2d$ and $\mathrm{div}(L)=t$. Then, the polarization of a generic pair in $\M_{d,t}^n$ is base-point-free.
\begin{proof}
Proposition \ref{prop locally free is an open property in families} shows that base-point-freeness is an open property in families of polarized hyperkähler manifolds. Therefore, as the moduli space $\M_{d,t}^n$
is irreducible by assumption, we just need to find an element $(\Tilde{X},\Tilde{L})\in \M_{d,t}^n$ for which the line bundle $\Tilde{L}$ is base-point-free. 
%To this end, we will deform the pair $(X,L)$ to a pair $(\Tilde{X},\Tilde{L})$ with $\Tilde{L}$ base-point-free and ample. Note that ampleness of $\Tilde{L}$ is needed to ensure that $(\Tilde{X},\Tilde{L})$ belongs to $\M_{d,t}^n$.
Note that $L$ is not assumed to be ample, so the pair $(X,L)$ does not necessarily belong to $\M_{d,t}^n$.
Let $\mathscr{X}\rightarrow S$ be a family of hyperkähler manifolds of generic Picard rank one such that
$\mathscr{X}_0\simeq X$ for some $0\in S$, and for which there exists a line bundle $\mathscr{L}\in\mathrm{Pic}(\mathscr{X})$ such that $\mathscr{L}|_{\mathscr{X}_0}=L$. By the projectivity criterion (see \cite[Thm.\ 3.11]{Huy} and \cite[Thm.\ 2]{HE}), the variety $\mathscr{X}_s$
is projective for all $s\in S$.
From Proposition \ref{prop locally free is an open property in families}, we conclude that the line bundle $\mathscr{L}_s$ is base-point-free and ample  for general $s\in S$. For this last step, we use the fact that the Picard rank of $\mathscr{X}_s$ is one for general $s\in S$, and the fact that $-\mathscr{L}_s$ cannot be ample since $\mathscr{L}_s$ has sections.
In particular, for general $s\in S$, the pair $(\mathscr{X}_s,\mathscr{L}_s)$ belongs to $\M^n_{d,t}$ and $\mathscr{L}_s$ is base-point-free. This concludes the proof.
\end{proof}
\end{prop}

%Let $n\in\{2,3,4\}$ and let $t,d>0$ be integers such that $\M^n_{d,t}$ is non-empty. Then, by Corollary \ref{cor:conn comp are irreducible} $\M^n_{d,t}$ is smooth. Since base-point-freeness is an open property in families of hyperkähler manifolds by Proposition \ref{prop locally free is an open property in families}, to prove generic base-point-freeness in the moduli space $\M^n_{d,t}$ it suffices to find an element $(X,L)\in\M^n_{d,t}$ with $L$ a base-point-free line bundle.
\section{Divisibility one}
\label{section divisibility one}

In this section, we prove generic base-point-freeness in the case where the divisibility of the polarization is one. The results of this section hold in all dimensions.
\begin{thm}
\label{thm bpf for divisibility 1}
The moduli space $\M^n_{d,1}$ is non-empty and irreducible for all $n>1$ and $d>0$, and the polarization of a generic pair in $\M^n_{d,1}$ is base-point-free.
\end{thm}

Before diving into the proof of Theorem \ref{thm bpf for divisibility 1}, let us recall some results on abelian varieties that we will use throughout this section, see \cite{BL} for all the necessary definitions and notations. A line bundle on an abelian variety is determined by its first Chern class and its semicharacter
by the Appell--Humbert theorem \cite[Thm.\ 1.2.3]{BL}. Denote by $\mathscr{L}(H,\chi)$ the unique line bundle whose first Chern class is $H$ and whose semicharacter is $\chi$. Then, the following results hold:
\begin{lem}\cite[Lem.\ 2.3.2]{BL}
\label{ref:lemma on line bundles and translations}
Let $T\coloneqq V/\Lambda$ be an abelian variety and let $\mathscr{L}(H,\chi)\in \mathrm{Pic}(T)$ be a line bundle on $T$. Then, for every $\overline{v}\in T$ with representative $v\in V$ the following holds
\[t_{\overline{v}}^*\mathscr{L}(H,\chi)=\mathscr{L}(H,\chi \mathrm{exp}(2\pi i \mathrm{Im}H(v,\cdot))),\]
where $t_{\overline{v}}$ is the translation on $T$ induced by the translation $t_v$ on $V$ which sends a vector $x \in V$ to the vector $v+x$.
\qed
\end{lem}

\begin{lem}\cite[Lem.\ 2.3.4]{BL}
\label{lem line bundle and pullback via homom.}
Let $f\colon T\rightarrow T'$ be a morphism of abelian varieties with analytic representation $F\colon V\rightarrow V'$. Then, for any $\mathscr{L}(H,\chi)\in \mathrm{Pic}(X')$,
\[\pushQED{\qed} f^*\mathscr{L}(H,\chi)=\mathscr{L}(F^*H,F^*\chi).\qedhere
\popQED\]
\end{lem}

Let us start the proof of Theorem \ref{thm bpf for divisibility 1} by proving that $\M_{d,1}^n$ is always non-empty:
\begin{lem}
\label{M_{d,1} is non-empty}
The space $\M_{d,1}^n$ is non-empty for every $n>1$ and $d>0$.
\begin{proof}
Let $T$ be an abelian surface which admits a polarization $L$ with $L^2=2d$ and let $L_n$ be the induced line bundle on the generalized Kummer variety $\mathrm{Kum}^n(T)$. This line bundle is not necessarily ample, but  satisfies $\mathrm{div}(L_n)=1$ and $q(L_n)=2d$. By deforming the pair $(\mathrm{Kum}^n(T),L_n)$ and using the projectivity criterion (see \cite[Thm.\ 3.11]{Huy} and \cite[Thm.\ 2]{HE}) as in the proof of Proposition \ref{prop sufficient condition to prove generic bpf}, we find a hyperkähler manifold belonging to $\M_{d,1}^n$. Thus, this moduli space is non-empty.
\end{proof}
\end{lem}

Note that Corollary \ref{cor:conn comp are irreducible} shows that $\M^n_{d,1}$ is irreducible for every $n$ since by Theorem \ref{thm conn comp of mod space} the moduli space $\M^n_{d,t}$ is connected for every $n$ if $t=1$.
Therefore, we are in the situation of Proposition \ref{prop sufficient condition to prove generic bpf}, and, to prove generic base-point-freeness in $\M^n_{d,1}$, it suffices to find a pair $(X,L)$ where $X$ is a $\mathrm{Kum}^{n}$-type variety and $L$ is a base-point-free line bundle on $X$ with $q(L)=2d$ and $\mathrm{div}(L)=1$. We do this by showing that, given any effective line bundle $L$ on an abelian surface $T$\label{effective assumption} the line bundle $L_n$ induced on the generalized Kummer variety $\mathrm{Kum}^n(T)$ is base-point-free, see Theorem \ref{thm: the line bundle on Kum is bpf}.
\smallskip

Let us consider the following commutative diagram:
\begin{equation}
\label{scheme}
\begin{tikzcd}
\mathrm{Kum}^n(T) \arrow[r, hook] \arrow[d] & T^{[n+1]} \arrow[r, "\sigma"] \arrow[d] & T \\
\overline{K^{n}} \arrow[r, hook]                & T^{(n+1)} \arrow[ru]                          &   \\
K^{n}_T \arrow[r, hook] \arrow[u]           & T^{n+1} \arrow[ruu, "\Sigma"'] \arrow[u]     &  
\end{tikzcd},
\end{equation}
where the first column is the base change by $0$ of the second column and $\sigma$ and $\Sigma$ are the natural maps induced by the summation map on $T$. By construction, $K_T^{n}=\ker(\Sigma)$ and the following holds:
\begin{lem}
\label{lem on the diagram}
Let $T$ be an abelian surface and let $K^{n}_T$ as in (\ref{scheme}). 
Then, the map
\begin{equation*}
    \varphi\colon  T^n\rightarrow K^{n}_T,\qquad (a_1,\ldots,a_n)\mapsto (a_1,\ldots,a_n,-a_1-\ldots-a_n)
\end{equation*}
is an isomorphism.
\qed
\end{lem}
\begin{rmk}
\label{rmk on S_3 action}
Note that $K^{n}_T$ is a subtorus of $T^{n+1}$ since it is the kernel of the morphism $\Sigma$. Moreover, the canonical action of the symmetric group $\mathfrak{S}_{n+1}$ on $T^{n+1}$ preserves $K^{n}_T$ and induces an action on it.
Identifying the torus $T^n$ with $K^{n}_T$ via the map $\varphi$, we see that the induced action of $\mathfrak{S}_{n+1}$ on $T^n$ is generated by the elements
\[
[(a_1,\ldots,a_n)\mapsto (a_i,a_2,\ldots,a_{i-1},a_1,a_{i+1},\ldots a_n)]\quad \text{for } i=1,\ldots, n, \quad \text{and}\]
\[[(a_1,\ldots,a_n)\mapsto (-a_1-\ldots-a_n,a_2,\ldots,a_n)].\]
\end{rmk}

Let $k$ be a positive integer. Given a line bundle $L$ on $T$, denote by $L^{\boxtimes k}\coloneqq L\boxtimes\ldots\boxtimes L$ the induced line bundle on $T^k$. Note that, via the map $\varphi$, the line bundle $L^{\boxtimes (n+1)}|_{K^{n}_T}$ is sent to
\[M\coloneqq (L^{\boxtimes n})\otimes (-1)^*\mu^*L,\]
where $\mu\colon T^n\rightarrow T$ is the addition map and $(-1)\colon T^n\rightarrow T^n$ is minus identity.
\smallskip

In the next proposition, we find line bundles $M_x$ on $T^n$ which are isomorphic to $M$:
\begin{prop}
\label{Prop M_x sim M}
Let $L$ be a line bundle on an abelian surface $T$, and 
let $M$ be as above. Then, for all $x\in T$, the line bundle
\[M_x\coloneqq t_{(nx,\ldots, nx)}^*(L^{\boxtimes n})\otimes t_{(-x,\ldots,-x)}^*(-1)^*\mu^*L\]
is isomorphic to the line bundle $M$.
\begin{proof}
By the Appell--Humbert theorem \cite[Thm.\ 1.2.3]{BL}, it suffices to check that the two line bundles $M$ and $M_x$ have the same first Chern class and the same semicharacter.
Let us assume that $L=\mathscr{L}(H,\chi)$ and denote by $\mathrm{pr}_i\colon T^n\rightarrow T$ the projection onto the $i$-th factor. By Lemma \ref{lem line bundle and pullback via homom.}, we have the following equalities:
\begin{align*}
    L&^{\boxtimes n}=\mathscr{L}\left(\displaystyle\sum_{i=1}^n \mathrm{pr}_i^*H,\prod_{i=1}^n \mathrm{pr}_i^*\chi\right),\\
&(-1)^*\mu^*L=\mathscr{L}\left(\mu^*H,\frac{1}{\mu^*\chi }\right).
\end{align*}

The second equality follows from Lemma \ref{lem line bundle and pullback via homom.} noting that $(-1)^*H=H$ and $(-1)^*\chi=\frac{1}{\chi}$.
%since semicharacters are morphism from the additive group $\Lambda$ to the multiplicative group $\mathbb{C}_1$).
We then deduce that
\[M=\mathscr{L}\left(\displaystyle\sum_{i=1}^n \mathrm{pr}_i^*H+\mu^*H,\frac{\prod_{i=1}^n \mathrm{pr}_i^*\chi}{\mu^*\chi}\right).\]
On the other hand, applying Lemma \ref{ref:lemma on line bundles and translations}, we see that
\[t_{(nx,\ldots,nx)}^*(L^{\boxtimes n})=\mathscr{L}\left(\displaystyle\sum_{i=1}^n \mathrm{pr}_i^*H,\prod_{i=1}^n \mathrm{pr}_i^*\chi\cdot e^{2\pi i \mathrm{Im}((\sum_{i=1}^n \mathrm{pr}_i^*H)((nx,\ldots,nx),\cdot))}\right),\]
\[
t_{(-x,\ldots,-x)}^*(-1)^*\mu^*L=\mathscr{L}\left(\mu^*H,\frac{e^{2\pi i \mathrm{Im}(\mu^*H((-x,\ldots,-x),\cdot))}}{\mu^*\chi}\right).
\]
Hence, the Hermitian forms of $M$ and $M_x$ coincide.
To see that also their semicharacters coincide just note that, for all $(a_1,\ldots,a_n)\in T^n$, the following holds:
\begin{align*}
   (\displaystyle\sum_{i=1}^n \mathrm{pr}_i^*H)((nx,&\ldots,nx),(a_1,\ldots,a_n))+\mu^*H((-x,\ldots,-x),(a_1,\ldots,a_n))=\\
   &=\displaystyle\sum_{i=1}^n H(nx,a_i)+H\left(-nx,\sum_{i=1}^n a_i\right)=0. \qedhere
\end{align*}
\end{proof}
\end{prop}

In the case where $L$ is effective, we aim to show that the line bundle $M$ on $T^n$ is generated by global sections which are symmetric with respect to the $\mathfrak{S}_{n+1}$-action on $T^n$.
\begin{lem}
\label{lem s_x are symmetric}
Let $L$ be an effective line bundle on an abelian surface $T$, and let $s\in H^0(L)$ be a non-zero section. For $x\in T$, consider the section $s_x$ of $M_x$ defined as follows:
\[s_x\coloneqq  t^*_{(nx,\ldots,nx)}\left(\displaystyle\bigotimes_{i=1}^n\mathrm{pr}_i^*s\right)\otimes  t^*_{(-x,\ldots,-x)}(-1)^*\mu^*s.\]
Then, $s_x$ is symmetric with respect to the $\mathfrak{S}_{n+1}$-action on $T^n$.
\begin{proof}
Let us compute $s_x(a_1,\ldots,a_n)$ for $(a_1,\ldots, a_n)\in T^n$:
\begin{align*}
    s_x(a_1,\ldots,a_n)  &=  \left(t^*_{(nx,\ldots, nx)}\left(\displaystyle\bigotimes_{i=1}^n\mathrm{pr}_i^*s\right)\otimes  t^*_{(-x,\ldots,-x)}(-1)^*\mu^*s\right)(a_1,\ldots,a_n)\\
    & = \left(\displaystyle\bigotimes_{i=1}^n\mathrm{pr}_i^*s\right)(a_1+nx,\ldots,a_n+nx)\cdot (-1)^*\mu^*s (a_1-x,\ldots, a_n-x)\\
    & = s(a_1+nx)\cdot \ldots\cdot s(a_n+nx)\cdot s(-a_1-\ldots-a_n+nx).
\end{align*}
Applying Remark \ref{rmk on S_3 action}, we deduce from this expression of $s_x(a_1,\ldots,a_n)$ that $s_x$ is symmetric with respect to the $\mathfrak{S}_{n+1}$-action on $T^n$.
\end{proof}
\end{lem}

\begin{prop}
\label{prop the induced line bundle on T^2 is bpf}
Let $L$ be an effective line bundle on an abelian surface $T$. Then, the induced line bundle $M$ on $T^n\simeq K^{n}_T$ is globally generated by sections which are symmetric with respect to the $\mathfrak{S}_{n+1}$-action inherited from $T^{n+1}$.
\begin{proof}
Let $s\in H^0(L)$ and let $D_T$ be the divisor on $T$ cut out by $s$. For every $x\in T$, let $D_x$ be the divisor on $T^n\simeq K^{n}_T$ cut out by the section $s_x$ of the line bundle \[M_x\coloneqq t_{(nx,\ldots,nx)}^*(L^{\boxtimes n})\otimes t_{(-x,\ldots,-x)}^*(-1)^*\mu^*L.\] In Proposition \ref{Prop M_x sim M}, we proved that the line bundle $M_x$ is isomorphic to $M$. Therefore, to prove base-point-freeness of $M$, it suffices to show that for every $(a_1,\ldots,a_n)\in T^n$ there exists an element $x\in T$ such that $(a_1,\ldots,a_n)\not \in D_x$.
Note that the following equivalences hold:
\begin{align*}
(a_1,\ldots,a_n)\in t_{(nx,\ldots,nx)}^*(D^{\boxtimes n}) \iff nx\in t^*_{a_i} D \text{ for some } i,
\\
(a_1,\ldots,a_n)\in t_{(-x,\ldots,-x)}^*(-1)^*\mu^*D \iff nx\in t^*_{-a_1-\ldots-a_n}D.
\end{align*}
Therefore, $(a_1,\ldots,a_n)\not \in D_x$ for any $x\in T$ such that $nx\not \in \textstyle\sum_i t^*_{a_i} D+t_{-a_1\ldots-a_n}^*D$. As the sections $s_x$ are symmetric with respect to the $\mathfrak{S}_{n+1}$-action on $T^n$ by Lemma \ref{lem s_x are symmetric}, we conclude that $M$ is generated by symmetric global sections.
\end{proof}
\end{prop}
We are now able to prove that the induced line bundle on the generalized Kummer variety is base-point-free:
\begin{thm}
\label{thm: the line bundle on Kum is bpf}
Let $n>1$, and let $L$ be an effective line bundle on an abelian surface $T$. Then, the induced line bundle $L_n$ on $\mathrm{Kum}^n(T)$ is base-point-free.
\begin{proof}
Since the line bundle $L^{\boxtimes (n+1)}$ on $T^{n+1}$ is symmetric with respect to the natural $\mathfrak{S}_{n+1}$-action, it descends to a line bundle $L_{(n+1)}$ on the symmetric product $T^{(n+1)}$. By Proposition \ref{prop the induced line bundle on T^2 is bpf}, there are symmetric global sections of the line bundle $L^{\boxtimes (n+1)}|_{K^{n}_T}\simeq M$ which generate it. These sections descend to sections of $L_{(n+1)}|_{\overline{K^{n}}}$ and they generate $L_{(n+1)}|_{\overline{K^{n}}}$. Therefore, the line bundle $L_{(n+1)}|_{\overline{K^{n}}}$ is base-point-free. Finally, since the line bundle $L_n$ on $\mathrm{Kum}^n(T)$ induced by the line bundle $L$ coincides with the pullback of the line bundle $L_{(n+1)}|_{\overline{K^{n}}}$ via the map $\mathrm{Kum}^n(T)\rightarrow \overline{K^{n}}$, we conclude that $L_{n}$ is base-point-free.
\end{proof}
\end{thm}
We can now conclude the proof of Theorem \ref{thm bpf for divisibility 1}:
\begin{proof}[Proof of Theorem \ref{thm bpf for divisibility 1}] For any $n>1$ and $d>0$, the moduli space $\M_{d,1}^n$ is non-empty by Lemma \ref{M_{d,1} is non-empty}. Let $(T,L)$ be an abelian surface such that $q(L)=2d$. Since $L$ is a polarization on $T$, the vector space $H^0(L)$ is non-empty. We can then apply Theorem \ref{thm: the line bundle on Kum is bpf} to conclude that the line bundle $L_n$ on $\mathrm{Kum}^n(T)$ is base-point-free. This, together with Proposition \ref{prop sufficient condition to prove generic bpf}, proves generic base-point-freeness in the moduli space $\M_{d,1}^n$.
\end{proof}

\section{Tautological bundles}
\label{Sec: tautological bundles}
In this section, we recall the definition of tautological bundles on the Hilbert scheme of points and the notion of $k$-very ampleness for a line bundle on an abelian surface. These will be used in Corollary \ref{if n-very ample then bpf on Kum} to deduce a criterion for base-point-freeness of the line bundles of the form $kL_n-\delta$ on $\mathrm{Kum}^n(T)$.
\smallskip

Let $T$ be an abelian surface and consider the universal family for $T^{[n+1]}$:
\[\Xi^{n+1}\coloneqq \{(P,\xi)\in T\times T^{[n+1]}\;|\; P\in\xi\}.\]
Denote by $p_T\colon \Xi^{n+1}\rightarrow T$ and by $p_{T^{[n+1]}}\colon \Xi^{n+1}\rightarrow T^{[n+1]}$ the two projections.
By construction, the fibre of $p_{T^{[n+1]}}$ over $\xi\in T^{[n+1]}$ is the subscheme $\xi\subseteq T$.
\begin{defn}
Let $L$ be a line bundle on an abelian surface $T$. The \textit{tautological bundle} associated to $L$ on $T^{[n+1]}$ is
\[L^{[n+1]}\coloneqq (p_{T^{[n+1]}})_*p_T^*(L).\]
\end{defn}
Since $p_{T^{[n+1]}}$ is a finite and flat morphism, $L^{[n+1]}$ is a vector bundle of rank $n+1$. Denoting by $\det L^{[n+1]}\coloneqq \bigwedge ^{n+1}L^{[n+1]}$ its determinant, the following equality holds:
\begin{equation}
\label{equation determinant}
 \det L^{[n+1]}=L_{n}-\delta,
\end{equation}
where, as in Remark \ref{rmk: delta on hilb}, $L_{n}$ is the line bundle that $L$ induces on $T^{[n+1]}$ and $\delta$ is the line bundle $\frac{1}{2}E$, see \cite[Lem.\ 3.7 \& Thm.\ 4.6]{L}.
The next lemma describes the global sections of the line bundles that we have introduced. For all the details, see \cite{SC}.

\begin{lem}
\label{lem global sections of these bundles}
With the previous notation, we have the following canonical isomorphisms:
\begin{itemize}
 \item[(\textit{i})] $H^0(T^{[n+1]}, L_{n})\simeq \mathrm{Sym}^{n+1} H^0(T, L)$;
 \item[(\textit{ii})] $H^0(T^{[n+1]}, L^{[n+1]})\simeq H^0(T, L);$
 \item[(\textit{iii})] $H^0(T^{[n+1]}, L_{n}-\delta)\simeq\bigwedge^{n+1} H^0(T,L).$
\end{itemize}
\end{lem}

\begin{defn}
A line bundle $L$ on an abelian surface $T$ is \textit{k-very ample} for some non-negative integer $k$ if, for every $\xi\in T^{[k+1]}$, the evaluation map
\[\mathrm{ev}_{L,\xi}\colon H^0(T,L)\twoheadrightarrow H^0(T, L\otimes \0_{\xi})\]
is surjective. 
\end{defn}
This notion is relevant for us due to the following result:
\begin{prop}
\label{prop: n-very ample iff globally gen}
Let $T$ be an abelian surface and let $L$ be a line bundle on $T$. With the previous notation, the following are equivalent:
\begin{itemize}
 \item[(\textit{i})] The line bundle $L_{n}-\delta$ on $T^{[n+1]}$ is globally generated.
 \item[(\textit{ii})] The tautological bundle $L^{[n+1]}$ on $T^{[n+1]}$ is globally generated.
 \item[(\textit{iii})] The line bundle $L$ is $n$-very ample.
\end{itemize}
\begin{proof}
Since $H^0(T^{[n+1]}, L^{[n+1]})\simeq H^0(T, L)$ by Lemma $\ref{lem global sections of these bundles}$, the equivalence of $(ii)$ and $(iii)$ is an immediate consequence of the fact that $p_{T^{[n+1]}}^*(\xi)\cong\xi$.
\smallskip

Let us prove the equivalence of $(i)$ and $(ii)$. It is straightforward to see that the vector bundle $L^{[n+1]}$ is globally generated if and only if the map
\[\textstyle\left(\bigwedge^{n+1}H^0(T^{[n+1]}, L^{[n+1]})\right)\otimes \0_{T^{[n+1]}}\rightarrow \bigwedge^{n+1}L^{[n+1]}\]
is surjective. Applying Lemma \ref{lem global sections of these bundles} and using $(\ref{equation determinant}$), this map is the same as the map 
\[H^0(T^{[n+1]},L_{n}-\delta)\otimes \0_{T^{[n+1]}}\rightarrow L_{n}-\delta.\]
Finally, this last map is surjective if and only if the line bundle  $L_{n}-\delta$ is globally generated.
\end{proof}
\end{prop}
As an immediate consequence of Proposition \ref{prop: n-very ample iff globally gen}, we have the following corollary:
\begin{cor}
\label{if n-very ample then bpf on Kum}
Let $T$ be an abelian surface and let $L$ be a line bundle on $T$. If the line bundle $L$ is $n$-very ample, then the line bundle $L_{n}-\delta$ on $\mathrm{Kum}^n(T)$ is globally generated.
\end{cor}
Let us end this section by recalling a criterion for k-very ampleness. Let $T$ be an abelian surface and denote by $\mathrm{Amp}(T)$ its ample cone. Let
\[f\colon\mathrm{Amp}(T)\rightarrow\mathbb{Z}_{\geq-1}\]
be the map that associates to an ample line bundle $L$ on $T$ the greatest integer $k$ such that $L$ is $k$-very ample; we set by definition $f(L)=-1$ if $L$ is not generated by global sections. Then, the following result holds.
\begin{prop}\cite[Thm.\ 4.3]{AM}
\label{proposition function about k very ampleness}
Let $(T, L)$ be an abelian surface such that $\mathrm{NS}(T)=\langle c_1(L)\rangle$ and let $q(L)=2d$. Then, for all $m\geq 2$,
\[f(mL)=2(m-1)d-2.\]
\end{prop}

\section{Other divisibilities}
\label{last section}
We have now all the tools needed to study generic base-point-freeness in the moduli space $\M^n_{d,t}$  for $n\in\{2,3,4\}$ when the divisibility is greater than one.
\smallskip

We begin by finding, for all possible values of $t>1$ and $d$ such that $\M^n_{d,t}$ is non-empty, a generalized Kummer variety together with a line bundle of divisibility $t$ and square $2d$. In the case $n=2$, the following holds:
\begin{prop}
\label{prop: existence an ab sur st the gen kum var belongs to M^2_d,t}
Let $t>1$ and $d>0$ be integers such that $\M^2_{d,t}$ is non-empty. Then, there exists an abelian surface $(T,L)$ such that the line bundle $tL_2-\delta$ on the hyperkähler manifold $\mathrm{Kum}^2(T)$ satisfies $q(tL_2-\delta)=2d$ and $\mathrm{div}(tL_2-\delta)=t$.

\begin{proof}
Since $n=2$, we have to consider only $t\in\{2,3,6\}$ by Remark \ref{rmk: divisibility divides GCD}. Given any primitive line bundle $L$ on an abelian surface $T$, the line bundle $tL_2-\delta$ on $\mathrm{Kum}^2(T)$ satisfies the condition $\mathrm{div}(tL_2-\delta)=t$. Therefore, for each given pair of integers $(d,t)$ such that $\M^2_{d,t}$ is non-empty, we just need to find an abelian surface such that $q(tL_2-\delta)=2d$.
\smallskip

Let $t=2$, and let $(X,H)\in \M^2_{d,2}$. By Proposition \ref{ref: prop lattice Kum type ISV}, there exists an abelian surface $(\Tilde{T},\Tilde{L})$ such that $H^2(X,\mathbb{Z})\cong H^2(\mathrm{Kum}^2(\Tilde{T}),\mathbb{Z})$ and, via this isomorphism,
\[H=a\Tilde{L}_2+b\delta,\]
for some integers $a$ and $b$. Note that $a$ and $b$ are coprime since $H$ is primitive.
Therefore, we deduce from $\mathrm{div}(H)=2$ that $a$ is even and $b$ is odd.
Write $a=2k$ for some integer $k$, and let $q(\Tilde{L})=2\Tilde{d}$. Define the following number:
 \begin{equation}
 \label{equation of d tilde for t=2}
 \hat{d}\coloneqq k^2\Tilde{d}-\frac{3(b^2-1)}{4}.
 \end{equation}
 Note that $\hat{d}$ is an integer since $b$ is odd. Let $T$ be an abelian surface which admits a polarization $L$ with $q(L)=2\hat{d}$. In particular, $q(L_2)=2\hat{d}$, and
 \[
 q(2L_2-\delta)=4(2\hat{d})-6=(2k)^2(2\Tilde{d})-6b^2=q(a\Tilde{L}+b\delta)=2d.\]
This allows us to conclude that $(\mathrm{Kum}^2(T), 2L_2-\delta)$ is a generalized Kummer variety with the prescribed invariants.
\smallskip

In the case of divisibility $t=3$ (resp., $t=6$), a similar argument works: Given an integer $d$ such that $\M^2_{d,3}$ (resp., $\M^2_{d,6}$) is non-empty, one sets $a, b,$ and $\Tilde{d}$ as in the previous case and checks that the number $\hat{d}\coloneqq k^2\Tilde{d}-\frac{(b^2-1)}{3}$ (resp., $\hat{d}\coloneqq k^2\Tilde{d}-\frac{(b^2-1)}{12}$) is an integer for any possible value of $b, k$ and $\Tilde{d}$. This allows us to conclude as in the previous case.
\end{proof}
\end{prop}

Similarly, in the case $n=3$, we have the following:
\begin{prop}
\label{prop: representatives in dimension 6}
Let $t>1$ and $d>0$ be integers such that the moduli space $\M_{d,t}^3$ is non-empty. Then, there exists an abelian surface $(T,L)$ such that one of the line bundles $tL_3-\delta$ and $8L_3-3\delta$ on the hyperkähler manifold $\mathrm{Kum}^3(T)$ has divisibility $t$ and square $2d$.
%\begin{enumerate}
%    \item $tL-\delta$ has divisibility $t$ and square $2d$ for some abelian surface $(T,L)$;
    %\item $n=3, t=8, d=64k-36$ for some $k>0$ and $8L-3\delta$ has divisibility $8$ and square $2d$ for some abelian surface $(T,L)$;
%    \item $n=4, t=5, d=5L-2\delta$ for some $k>0$ and $5L-2\delta$ has divisibility $5$ and square $2d$ for some abelian surface $(T,L)$;
 %   \item $n=4, t=10, d=100k-45$ for some $k>0$ and $10L-3\delta$ has divisibility $10$ and square $2d$ for some abelian surface $(T,L)$.
%\end{enumerate}
\begin{proof}
Since $n=3$, $\delta$ has square $2n+2=8.$ By Remark \ref{rmk: divisibility divides GCD}, the only possible values for $t$ are $t=1,2,4,8.$ As in the proof of Proposition \ref{prop: existence an ab sur st the gen kum var belongs to M^2_d,t}, it suffices to find the correct value for the square $2\hat{d}$ of the line bundle $L$, for any fixed value of $d$ and $t$. Given $d$ and $t$, let $a,b$, and $\Tilde{d}$ be the integers as in the proof of Proposition \ref{prop: existence an ab sur st the gen kum var belongs to M^2_d,t}.
\smallskip

When $t=2$ (resp., $t=4$), set $\hat{d}\coloneqq \frac{a^2\Tilde{d}}{2}-b^2+1$ (resp., $\hat{d}=\frac{a^2\Tilde{d}}{16}-\frac{b^2-1}{4}$). Then, $\hat{d}$ is an integer, and, given an abelian surface $(T,L)$ such that $q(L)=2\hat{d}$, the line bundle $2L_3-\delta$ (resp., $3L_3-\delta$) on the generalized Kummer variety $\mathrm{Kum}^3(T)$ has square $2d$ and divisibility $2$ (resp., $3$).
\smallskip

For $t=8$, we have that $\mathrm{gcd}(a,8b)=8$. In particular, $a$ is even, and, as $a$ and $b$ are coprime, $b$ is odd. The square $b^2$ modulo $16$ is then either $1$ or $9$. If $b^2\equiv 1$ modulo $16$, set $\hat{d}=\frac{a^2\Tilde{d}}{64}-\frac{b^2-1}{16}$. Given an abelian surface $(T,L)$ such that $q(L)=2\hat{d}$, the line bundle $8L_3-\delta$ on $\mathrm{Kum}^3(T)$ has then square $2d$ and divisibility $2$. If $b^2\equiv 9$ modulo $16$, set $\hat{d}=\frac{a^2\Tilde{d}}{64}-\frac{b^2-9}{16}$. Then, given an abelian surface $(T,L)$ such that $q(L)=2\hat{d}$, the line bundle $8L_3-3\delta$ on $\mathrm{Kum}^3(T)$ has square $2d$ and divisibility $8$.
\end{proof}
\end{prop}
\begin{rmk}
\label{rmk: representatives in dimension 6}
Let $t>1$ and let $d>0$ be integers such that the moduli space $\M_{d,t}^3$ is non-empty. From the proof of Proposition \ref{prop: representatives in dimension 6}, we can also deduce the following: When $t=8$ and $d=64k-36$ for some positive integer $k$, it is possible to find an abelian surface $(T,L)$ for which $8L-3\delta$ has square $2d$, but it is not possible to find an abelian surface $(\Tilde{T},\Tilde{L})$ with $q(8\Tilde{L}_3-\delta)=2d.$ In all the other cases, a representative of the form $(tL_3-\delta)$ can be found.
\end{rmk}
Finally, in the case $n=4$, one proves the following result using the same techniques as in Proposition \ref{prop: representatives in dimension 6}.

\begin{prop}
\label{prop: representatives in dimension 8}
Let $t>1$, and let $d>0$ be integers such that the moduli space $\M_{d,t}^4$ is non-empty. Then, there exists an abelian surface $(T,L)$ such that one of the following line bundles $tL_4-\delta$, $5L_4-2\delta$, and $10L_4-3\delta$ on the hyperkähler manifold $\mathrm{Kum}^4(T)$ has divisibility $t$ and square $2d$.
\qed
\end{prop}
Proposition \ref{prop: existence an ab sur st the gen kum var belongs to M^2_d,t} allows us to prove generic base-point-freeness in $\M_{d,t}^2$ if $(d,t)\not=(1,2)$: 
\begin{thm}
\label{thm bpf for n=2}
Let $t$ and $d$ be positive integers such that $\M_{d,t}^2$ is non-empty. Assume that $(d,t)\not=(1,2)$. Then, for a generic pair $(X,L)\in\M_{d,t}^2$ the line bundle $L$ is base-point-free.
\begin{proof}
Let $(t,d)$ be a pair of integers such that $\M_{d,t}^2$ is non-empty. By Proposition \ref{prop sufficient condition to prove generic bpf}, to prove generic base-point-freeness in $\M_{d,t}^2$, it suffices  to find a pair $(X,L)$ where $X$ is a hyperkähler manifold of $\mathrm{Kum}^2$-type and $L$ is a base-point-free line bundle (not necessarily ample) on $X$ such that $q(L)=2d$ and $\mathrm{div}(L)=t$.
By Remark \ref{rmk: divisibility divides GCD}, the only possible values for the divisibility $t$ are $t=1,2,3$, and $6$.
\smallskip

As the case $t=1$ follows from Theorem \ref{thm bpf for divisibility 1}, let $t$ be equal to either $2, 3$, or $6$, and let $d$ be a positive integer such that $\M_{d,t}^2$ is non-empty. Let $(T,L)$ be an abelian surface such that $q(tL_2-\delta)=2d$. 
The existence of such an abelian surface is guaranteed by Proposition \ref{prop: existence an ab sur st the gen kum var belongs to M^2_d,t}.
Recall that a generic abelian surface has Picard number $1$ (this can be deduced from the analysis presented in \cite[Sec.\ 2.7]{BL2}). Therefore, by deforming $(T,L)$, we may assume NS$(T)=\langle c_1(L)\rangle$. Write $2\hat{d}=q(L)$, and let $f$ be the function of Proposition \ref{proposition function about k very ampleness}:
\begin{equation*}
 \label{equation a}
 f(tL)=2(t-1)\hat{d}-2.
\end{equation*}
Since we are assuming $t>1$, we deduce that $f(tL)\geq2$ for all the cases that we are considering: Note that, in the case $t=2$, we have $\hat{d}\geq 2$ since we are assuming $(d,t)\not = (1,2)$. By definition of $f$, this implies that the line bundle $tL$ is $2$-very ample. Applying Corollary \ref{if n-very ample then bpf on Kum}, we deduce that the line bundle $tL_2-\delta$ on $\mathrm{Kum}^2(T)$ is base-point-free. This provides the required element and concludes the proof in the case $n=2$.
\end{proof}
\end{thm}
Similarly, using Proposition \ref{rmk: representatives in dimension 6}, it is possible to prove generic base-point-freeness in $\M_{d,t}^3$ in all but three cases:
\begin{thm}
\label{thm: bpf for n=3}
Let $t$ and $d$ be positive integers such that $\M_{d,t}^3$ is non-empty. Assume that $(d,t)\not=(4,2),(28,8),(92,8)$. Then, the polarization of a generic pair in $\M_{d,t}^3$ is base-point-free.
\begin{proof}
As the case $t=1$ follows from Theorem \ref{thm bpf for divisibility 1}, let $t>1$, and let $d$ be a positive integer such that $\M_{d,t}^3$ is non-empty. 
Write $2\hat{d}=q(L)$, and let $f$ be the function of Proposition \ref{proposition function about k very ampleness}:
\begin{equation}
 \label{eq: a}
 f(tL)=2(t-1)\hat{d}-2.
\end{equation}
Note that $f(tL)\geq 3$ for all the cases we are considering: This is immediate when $t=4$ and $t=8$. When $t=2$ and $\hat{d}=1$, then $2d=q(2L_3-\delta)=0$, which is not possible since $d$ is a positive integer, and when $\hat{d}=2$, then $2d=q(2L_3-\delta)=8$ and this is the case we excluded.
In particular, in the cases we are considering, the line bundle $L$ is $3$-very ample. Therefore, the line bundle $tL_3-\delta$ on $\mathrm{Kum}^3(T)$ is base-point-free. In this case, generic base-point-freeness follows from Proposition \ref{prop sufficient condition to prove generic bpf}.
\smallskip

To conclude the proof, we need to study the cases for which there exists no abelian surface $(T,L)$ for which $tL_3-\delta$ on $\mathrm{Kum}^3(T)$ has square $2d$. By Proposition \ref{prop: representatives in dimension 6} and Remark \ref{rmk: representatives in dimension 6}, this happens only when $t=8$ and $d=64k-36$ for some positive integer $k$, and in this case, there exists an abelian variety $(T,L)$ such that the line bundle $8L_3-3\delta$ on $\mathrm{Kum}^3(T)$ has square $2d$ and divisibility $8$. In particular, $q(8L_3-3\delta)=128k-72$, then $q(L)=2k$. Write $8L_3-3\delta=(4L_3-\delta)+2(2L_3-\delta)$. The line bundle $4L_3-\delta$ is base-point-free by the proof of the case of divisibility $4$. On the other hand, the line bundle $2L_3-\delta$ has square $8k-8$.
Since we are assuming $(d,t)\not=(28,8),(92,8)$, we deduce that $k>2$. Therefore, $8k-8>8$. By using the proof in the case of divisibility $2$, we conclude that $2L_3-\delta$ is base-point-free, since we are not in the case $(d,t)=(4,2)$. Therefore, the line bundle $8L_3-3\delta$ is base-point-free being the sum of line bundles which are base-point-free. This concludes the proof.
\end{proof}
\end{thm}
The same technique can be applied to prove generic base-point-freeness in the moduli space $\M_{d,t}^4$ for almost all pairs $(d,t)$:
\begin{thm}
\label{thm: bpf for n=4}
Let $t$ and $d$ be positive integers such that $\M_{d,t}^4$ is non-empty. Assume that $(d,t)\not=(3,2),(20,5),(55,10)$. Then, the polarization of a generic pair in $\M_{d,t}^4$ is base-point-free.
\end{thm}
The results of this section can be summarized as follows:
\begin{thm}
\label{Final Theorem}
Let $n\in\{2,3,4\}$, and let $d$ and $t$ be positive integers such that the moduli space $\M_{d,t}^n$ is non-empty. Let $A$ be the set of triples
\[A\coloneqq\{(2,1,2),(3,4,2),(3,28,8),(3,92,8),(4,3,2),(4,20,5),(4,55,10)\}.\] 
Then, if $(n,d,t)\not\in A$, the polarization of a general pair in $\M_{d,t}^n$ is base-point-free.
\qed
\end{thm}

\end{document}